\documentclass[10pt,onecolumn]{article}

\usepackage{graphicx}
\usepackage{epsfig}
\usepackage{amssymb,amsmath,amsfonts,amsthm}
\usepackage{algorithm}
\usepackage{algpseudocode}
\usepackage{epstopdf}
\usepackage{url}
\usepackage{authblk}

\theoremstyle{plain}
\newtheorem{lemma}{Lemma}

\newcommand{\pmat}[1]{\begin{pmatrix}#1\end{pmatrix}} 
\newcommand{\Real}{\ensuremath{\mathbb{R}}} 
\newcommand{\Complex}{\ensuremath{\mathbb{C}}} 

\newcommand{\dia}[1]{\ensuremath{\diagm\left(#1\right)}} 
\newcommand{\adi}[1]{\ensuremath{\adiagmm\left(#1\right)}} 

\newcommand{\trace}[1]{\ensuremath{\tr\left( #1 \right)}} 
\newcommand{\rank}[1]{\ensuremath{\ra\left( #1 \right)}} 
\newcommand{\norm}[1]{\ensuremath{\left\| #1 \right\|}} 
\newcommand{\abs}[1]{\ensuremath{\left| #1 \right|}}
\newcommand{\inp}[2]{\ensuremath{\langle #1 , #2 \rangle}} 

\newcommand{\vect}[1]{\ensuremath{\ve\left( #1 \right)}}
\newcommand{\expp}[1]{\ensuremath{e^{ #1}}}
\newcommand{\fourier}[1]{\ensuremath{\mathcal{F}\left\{ #1 \right\}}}

\newcommand{\bigO}[1]{\ensuremath{\mathcal{O}\left( #1 \right)}}
\newcommand{\cc}{j}

\newcommand{\re}[1]{ \mathcal{R}({#1}) }
\newcommand{\im}[1]{ \mathcal{I}({#1}) }

\newcommand{\aX}{\pmb{X}}
\newcommand{\aY}{\pmb{Y}}
\newcommand{\aZ}{\pmb{Z}}

\newcommand{\vecto}{{\ve}}
\newcommand{\veca}{{\mathbf{a}}}
\newcommand{\vecb}{{\mathbf{b}}}

\newcommand{\vecy}{{\mathbf{y}}}
\newcommand{\vecu}{{\mathbf{u}}}

\newcommand{\vecx}{{\mathbf{x}}}
\newcommand{\px}{{u}}
\newcommand{\py}{{v}}
\newcommand{\aperture}{{\mathbf{A}}}

\DeclareMathOperator*{\tr}{trace}
\DeclareMathOperator*{\ra}{rank}
\DeclareMathOperator*{\diagm}{d}

\DeclareMathOperator*{\adiagmm}{\widehat{d}}
\DeclareMathOperator*{\ve}{vect}
\DeclareMathOperator*{\argmin}{\arg\min}

\newtheorem{thm}{Theorem}[section]
\newtheorem{proposition}[thm]{Proposition}

\newtheorem{remark}[thm]{Remark}
\newcommand{\paren}[1]{\left(#1\right)}

\newcommand{\parennn}[1]{\left\{#1\right\}}

\DeclareMathOperator{\Id}{Id}
\DeclareMathOperator{\Fix}{\mathsf{Fix}\,}
\DeclareMathOperator{\dist}{dist}

\title{Solving large-scale general phase retrieval problems via a sequence of convex relaxations}

\author[1,*]{Reinier Doelman}
\author[1]{H.~Thao Nguyen}
\author[1]{Michel Verhaegen}

\affil[1]{Delft Center for Systems and Control, Delft University of Technology, The Netherlands}

\affil[*]{Corresponding author: r.doelman@tudelft.nl}

\date{}

\begin{document}

\maketitle

\begin{abstract}
We present a convex relaxation-based algorithm for large-scale general phase retrieval problems.
    General phase retrieval problems include i.a. the estimation of the phase of the optical field in the pupil plane based on intensity measurements of a point source recorded in the image (focal) plane.
    The non-convex problem of finding the complex field that generates the correct intensity is reformulated into a rank constraint problem.
    The nuclear norm is used to obtain the convex relaxation of the phase retrieval problem.
    A new iterative method, indicated as \emph{Convex Optimization-based Phase Retrieval} (COPR), is presented, with each iteration consisting of solving a convex problem.
    In the noise-free case and for a class of phase retrieval problems the solutions of the minimization problems converge linearly or faster towards a correct solution.
    Since the solutions to nuclear norm minimization problems can be computed using semidefinite programming, and this tends to be an expensive optimization in terms of scalability, we provide a fast ADMM algorithm that exploits the problem structure.
    The performance of the COPR algorithm is demonstrated in a realistic numerical simulation study, demonstrating its improvements in reliability and speed with respect to state-of-the-art methods.
\end{abstract}

\section{Introduction}
Recovery of a signal from several measured intensity patterns, also known as the \emph{phase retrieval problem}, is of great interest in optics and imaging.
Recently it was shown in \cite{Antonello15} that the problem of estimating the wavefront aberration from measurements of the point spread functions can be formulated as a phase retrieval problem.

In this paper, we consider the general phase retrieval problem \cite{SheEld15}:
\begin{equation*}
\mbox{find}\quad \mathbf{a} \in \mathbb{C}^{n_a} \mbox{ such that }
\vecy_i = | \vecu_i^H \veca |^2 \quad {\rm for}\;\; i=1,\ldots,n_y,
\end{equation*}
where $\vecy_i \in \mathbb{R}_+$ and $\vecu_i \in \mathbb{C}^{n_a}$ are known and $(\cdot)^H$ denotes the Hermitian transpose of a vector (matrix).
For brevity the following compact notation will be used in this paper to denote this general noise-free  phase retrieval problem:
\begin{equation}\label{ProblemG}
\mbox{find}\quad \mathbf{a} \in \mathbb{C}^{n_a} \mbox{ such that }
\vecy = | U \veca |^2,
\end{equation}
where $\vecy \in \Real_+^{n_y}$ are the measurements and $U \in \Complex^{n_y\times n_a}$ is the propagation matrix.
With noise on the measurements $y_i$, we consider the following related optimization problem:
\begin{equation}
	\begin{aligned}
    	& \underset{\veca \in \Complex^{n_a}}{\min}
        & & \norm{\vecy - \abs{U\veca}^2},
        \end{aligned} \label{eq:problemG2}
\end{equation}
where $\norm{\cdot}$ denotes a vector norm of interest.

The sparse variant of the phase retrieval problem corresponds to the case that the unknown parameter $\mathbf{a}$ is a sparse vector.
A special case of this problem is when the measurements are the magnitude of the Fourier transform of multiples of $\mathbf{a}$ with certain phase diversity patterns.
A number of algorithms utilizing the Fourier transform have been proposed for solving this class of phase retrieval problems \cite{fienup1982phase,LukBurLyo02,Gespar}.

The fundamental nature of \eqref{ProblemG} has given rise to a wide variety of solution methods that have been developed for specific variants of this problem since the observation of Sayre in 1952 that phase information of a scattered wave may be recovered from the recorded intensity patterns at and between Bragg peaks of a diffracted wave \cite{Sayre52}.
Direct methods \cite{Hauptman86} usually use insights about the crystallographic structure and randomization to search for the missing phase information.
The requirement of such a-priori structural information and the expensive computational complexity often limit the application of these methods in practice.

A second class of methods first devised by Gerchberg and Saxton \cite{GSF72} and Fienup \cite{fienup1982phase} can be described as variants of the method of alternating projections on certain sets defined by the constraints. For an overview of these methods and latter refinements we refer the reader to \cite{Bauschke02,LukBurLyo02}.

In \cite{candes2015phase} \eqref{ProblemG} is relaxed to a convex optimization problem.
The inclusion of the sparsity constraint in the same framework of convex relaxations has been considered in \cite{Ohlsson}.
However, as reported in \cite{Gespar} the combination of matrix lifting and semidefinite programming (SDP) makes this method not suitable for large-scale problems.
To deal with large-scale problems, the authors of \cite{Gespar} have proposed an iterative solution method, called GESPAR, which appears to yield promising recovery of very sparse signals.
However, this method consists of a heuristic search for the support of $\veca$ in combination with a variant of Gauss-Newton method, whose computational complexity is often expensive.
These algorithmic features are potential drawbacks of GESPAR.

In this paper, we propose a sequence of convex relaxations for the phase retrieval problem in \eqref{ProblemG}.
Contrary to existing convex relaxation schemes such as those proposed in \cite{candes2015phase,Ohlsson}, matrix lifting is not required in our strategy. The obtained convex problems are affine in the unknown parameter vector $\veca$.
Contrary to \cite{candes2013phaselift}, our strategy does not require the tuning of regularization parameters when the measurements are corrupted by noise.
We then present an ADMM-based algorithm that can solve the resulting optimization problems effectively.
This potentially addresses the restriction of current SDP-based methods to only relatively small-scale problems.

In Section~\ref{sec:wavefrontestimation} we formulate the estimation problem of our interest for both zonal and modal forms.
In Section~\ref{sec:algorithm} we propose an algorithm for solving this problem.
Since this algorithm is based on minimizing a nuclear norm, a computationally heavy minimization problem, we suggest an ADMM-based algorithm in Section~\ref{sec:admm} that exploits the problem structure.
This ADMM algorithm features two minimization problems whose solutions can be computed exactly and with complexity $\bigO{n_y n_a}$, where $n_y$ is the number of measurements and $n_a$ is the number of unknown variables.

Analytic solutions for the ADMM algorithm update steps will be presented in Subsections~\ref{sec:aupdate} and \ref{sec:Xupdate}.
The convergence behaviour of the algorithm proposed in Section~\ref{sec:algorithm} is analysed in Section~\ref{sec:convergence}.
In Sections~\ref{sec:numericalexperiments} we describe and discuss the results of a number of numerical experiments that demonstrate the promising performances of our algorithms.
We end with concluding remarks in Section~\ref{sec:remarks}.

\section{Wavefront estimation from intensity measurements}\label{sec:wavefrontestimation}
The problem of phase retrieval from the point spread function images can be approached from 2 directions. We first describe the problem in zonal form, and then in modal form.

\subsection{Problem formulation in zonal form}
\label{subsec:zonal form}

In \cite{Antonello15} it was shown that reconstructing the wavefront from CCD recorded images of a point source may also be formulated as a phase retrieval problem.
These recorded images are called {\em point spread functions (PSFs)}.
As such approaches avoid the requirement of extra hardware to sense the wavefront, such as a Shack-Hartmann wavefront sensor, the problem is relevant and summarized here.

The PSF is derived from the magnitude of the Fourier transform of the generalized pupil function (GPF).
For an aberrated optical system the GPF is defined as the complex valued function \cite{goodman2008introduction}:
\begin{equation}
 \label{eq:GPF}
  P(\rho,\theta) = \aperture(\rho,\theta)\expp{\cc \phi(\rho,\theta)},
\end{equation}
where $\rho$ (radius) and $\theta$ (angle) specify the normalized polar coordinates in the exit pupil plane of the optical system.
In \eqref{eq:GPF}, $\mathbf{A}(\rho,\theta)$ is the amplitude apodisation function and $\phi(\rho,\theta)$ is the phase aberration function.

The aim of the wavefront reconstruction problem is to estimate $\phi(\rho,\theta)$.
Once this phase aberration of an optical system has been estimated, it can be corrected by using phase modulating devices such as deformable mirrors.

In order to estimate $\phi(\rho,\theta)$, a known phase diversity pattern $\phi_d(\rho,\theta)$ can be introduced (e.g., by using a deformable mirror) to transform the GPF in a controlled manner into the aberrated GPF:
\begin{equation}\label{eq:GPFd}
  P_d(\rho,\theta)
  = \mathbf{A}(\rho,\theta) \expp{\cc \phi(\rho,\theta)} \expp{\cc \phi_d(\rho,\theta)}.
\end{equation}
The noise-free intensity pattern of  $P_d(\rho,\theta)$ measured at the image plane is denoted
\begin{equation}\label{eq:Intensity_d}
\vecy_d = \abs{\fourier{ \aperture(\rho,\theta) \expp{\cc \phi(\rho,\theta)} \expp{\cc \phi_d(\rho,\theta)} }  }^2.
\end{equation}
If we sample the function $P_d(\rho,\theta)$ at points corresponding to a square grid of size $m \times m$ on the pupil plane, then  $\aperture(\rho,\theta)$, $\phi_d(\rho,\theta)$ and $\phi(\rho,\theta)$ are square matrices of that size.

Let us define $\vecto(\cdot)$ the vectorization operator such that $\vecto(Z)$ yields the vector obtained by stacking the columns of matrix $Z$ into a column vector.
The inverse operator $\vecto^{-1}(\cdot)$, which maps a column vector of size $m^2$ to a square matrix of size $m \times m$, is also well defined. Let in particular the matrix $Z$ and the vector $\veca$ be defined as:
\begin{equation*}
	Z = \aperture(\rho,\theta) e^{j\phi(\rho,\theta)} \in \mathbb{C}^{m \times m},\quad  \veca = \vecto(Z) \in \mathbb{C}^{m^2}. \label{eq:unknownpupilF}
\end{equation*}
With the definition of the vector $\mathbf{p}_d$:
\begin{equation*}
	\mathbf{p}_d  = \vect{e^{j\phi_d(\rho,\theta)}}\in \mathbb{C}^{m^2},
\end{equation*}
and with $D_d = \dia{\mathbf{p}_d} \in \mathbb{C}^{m^2\times m^2} $ the diagonal matrix with diagonal entries taken from the vector $\mathbf{p}_d$, we can write the noise-free intensity measurements in \eqref{eq:Intensity_d} as
\begin{equation*}
     \vecy_d = \abs{ \fourier{ \expp{\cc \phi_d(\rho,\theta)} Z}  }^2
     = \abs{ \fourier{ \vecto^{-1}(D_d\mathbf{a})}  }^2.
\end{equation*}
As the Fourier transform is a linear operator, we can write our noise-free intensity measurements in the form:
\begin{equation}\label{eq:Intensity_f}
	\vecy_d = \left| U_d \veca \right|^2,
\end{equation}
where in this case $U_d$ is a unitary matrix.

By stacking the vectors $\vecy_d$ and the matrices $U_d$, obtained from the $n_d$ images with $n_d$ different phase diversities, correspondingly into the vector $\vecy$ and the matrix $U$ (of size $n_d m^2 \times m^2$), the problem of finding $\veca$ from noise-free intensity measurements can be formulated as in \eqref{ProblemG} and that from noisy measurements can be formulated as in \eqref{eq:problemG2} for $n_a=m^2$ and $n_y=n_dm^2$.

It is worth noting that the dimension of the unknown $\veca$ with $m$ in the range of a couple of hundreds turns this problem into a non-convex large-scale optimization problem. For such  a problem the implementation of PhaseLift \cite{candes2013phaselift} using standard semidefinite programming, using libraries like MOSEK \cite{mosek}, will not be tractable because of the large matrix dimensions of the unknown quantity.
If we assume that the computational complexity of semidefinite programming with matrix constraints of size $n \times n$ increases with $\bigO{n^6}$ \cite{vandenberghe2005interior}, then a naive implementation of the PhaseLift method applied to \eqref{eq:problemG2} involving a single image has worst-case computational complexity of $\bigO{m^{12}}$.

\subsection{Problem formulation in modal form}\label{sec:modal}
In general, only approximate solutions can be expected for a phase retrieval problem.
In the modal form of the phase retrieval problem, also considered in \cite{Antonello15} for extended Nijboer-Zernike (ENZ) basis functions, the GPF is assumed to be well approximated by a weighted sum of basis functions.
We make use of real-valued radial basis functions \cite{martinez2016computation} with complex coefficients to approximate the GPF. These are studied in the scope of wavefront estimation in \cite{Piet17} and an illustration of these basis function on a $4 \times 4$ grid in the pupil plane is given in Figure~\ref{fig:bf}.
\begin{figure}[ht]
	\centering
    \includegraphics[width=0.5\columnwidth]{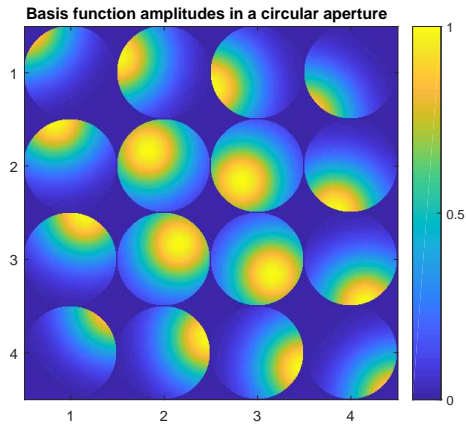}
    \caption{16 radial basis functions with centers in a $4 \times 4$ grid, with circular aperture support.}
    \label{fig:bf}
\end{figure}

Switching from the polar coordinates $(\rho,\theta)$ to the Cartesian coordinates $(x,y)$ in the pupil plane, let us consider the radial basis functions and the approximate GPF given by
\begin{equation}
	\begin{aligned}
	G_i(x,y)	&=
    \chi(x,y)\expp{-\lambda_i \left((x-x_i)^2 + (y-y_i)^2\right) }, \\
    P(x,y) 		&\approx \widetilde{P}(x,y,\veca) = \sum_{i=1}^{n_a} a_i G_i(x,y),
    \end{aligned} \label{eq:basis}
\end{equation}
where $(x_i,y_i)$ are the centers of basis functions $G_i(x,y)$, $a_i \in \Complex$, $\lambda_i \in \Real_+$ determines the spread of that function, $\chi(x,y)$ denotes the support of the aperture, and $\veca$ is the coefficient parameter vector to be estimated. The parameters $\lambda_i$ are usually taken equal for all basis functions and for their tuning we refer to \cite{Piet17}.

The aberrated GPF corresponding to the introduction of phase diversity $\phi_d$ is
\begin{equation}\label{eq:aGPFd}
	\widetilde{P}_d(x,y,\veca,\phi_d) = \sum_{i=1}^{n_a} a_i G_i(x,y) \expp{j\phi_d(x,y)}.
\end{equation}

The normalized complex PSF is the 2-dimensional Fourier transform of the GPF
\cite{janssen2002extended, braat2002assessment}.
The aberrated PSF corresponding to the aberrated GPF in \eqref{eq:aGPFd} is given as
\begin{equation}\label{eq:aPSF}
	p_d(\px,\py) = \sum_{i=1}^{n_a} a_i \fourier{G_i(x,y) \expp{j\phi_d(x,y)} } = \sum_{i=1}^{n_a} a_i U_{d,i}(\px,\py),
\end{equation}
where $(\px,\py)$ are the Cartesian coordinates in the image plane of the optical system.

We now drop the dependency on the  coordinates and vectorize expression \eqref{eq:aPSF} for all $n_d$ diversities that have been applied to obtain the following compact form of a single matrix-vector multiplication,
\begin{equation}\label{p=Ua}
 \mathbf{p} = U \veca.
\end{equation}
The vector $\mathbf{p}$ is the obtained vectorization and combination over all the aberrated PSFs, and the matrix $U$ is the vectorized and concatenated version of the functions
 $U_{d,i}$ sampled on a grid of size $m \times m$.

Let the intensity of the PSFs be recorded on the corresponding grid of pixels of size $m \times m$, and let the vectorization of this intensity pattern for different phase diversities be concatenated into the vector $\vecy$.
We can again formulate the problem of finding $\veca$ from noise-free intensity measurements as in \eqref{ProblemG} and from noisy measurements as in \eqref{eq:problemG2}
for $n_y = m^2n_d$.

It is worth noting that the dimension of $\veca$ is not dependent on the size of the sample grid (the size of the problem).
This is the fundamental advantage of the modal form formulation over the zonal form one, for which the size of $\veca$ directly depends on the size of the problem, i.e. $n_a=m^2$.

In this paper two steps are combined to deal with the large-scale nature of optimization \eqref{eq:problemG2}:
\begin{enumerate}
	\item The unknown pupil function $P(\rho,\theta)$ can be represented as a linear combination of a number of basis functions.
    In \cite{Antonello15} use has been made of the ENZ basis functions, while in \cite{Piet17} use is made of radial basis functions instead of ENZ ones.
    The radial basis functions are used here as \cite{Piet17} demonstrated their advantages over the ENZ type.
 	\item A new strategy is proposed for solving optimization \eqref{ProblemG} via a sequence of convex optimization problems.
    Each of the subproblems can be solved effectively by an iterative ADMM algorithm that exploits the problem structure.
\end{enumerate}

In the following we assume that the problem is normalized such that all entries of $\vecy$ have values between 0 and 1.

\section{The COPR algorithm}\label{sec:algorithm}

Equation \ref{ProblemG} is equivalent to a rank constraint.
Define the matrix-valued function
\begin{equation}
	M(A,B,C,X,Y) = \pmat{C+AY+XB+XY & A+X \\ B+Y & I}, \label{eq:M}
\end{equation}
where $I$ is the identity matrix of appropriate size. 
Let $\vecb \in \Complex^{n_a}$ be a coefficient vector.
For notational convenience, we will denote
\begin{equation*}
	\begin{aligned}
	&M(U,\veca,\vecb,\vecy) = \\ 
    &M\left(\dia{\veca^HU^H},\dia{U\veca},\dia{\vecy},\dia{\vecb^HU^H},\dia{U\vecb}\right).
	\end{aligned}
\end{equation*}

Our proposed algorithm in this paper relies on the following fundamental result.

\begin{lemma}\label{lem:rank}\cite{doelman2016sequential}
For any $\vecb \in \Complex^{n_a}$, the constraint $\vecy = \abs{U\veca}^2$ is equivalent to the constraint
\begin{equation*}
	\rank{M(U,\veca,\vecb,\vecy)} = n_y.
\end{equation*}
\end{lemma}

For addressing problem \eqref{eq:problemG2}, Lemma~\ref{lem:rank} suggests a consideration of the following approximate problem, for a user-selected parameter vector $\vecb$,
\begin{equation} \label{min rank}
	\min_{\veca \in \Complex^{n_a}} \rank{M(U,\veca,\vecb,\vecy)}.
\end{equation}

Since \eqref{min rank} is a non-convex problem and to anticipate the presence of measurement noise, we propose to solve the following convex optimization problem:
\begin{equation}
	\min_{\veca \in \Complex^{n_a}} f(\veca) := \norm{M(U,\veca,\vecb,\vecy)}_*, \label{EQ:NN}
\end{equation}
where $\norm{\cdot}_*$ denotes the nuclear norm of a matrix, the sum of its singular values \cite{recht2010guaranteed}.

In the case that prior knowledge on the problem indicates that
$\veca$ is a sparse vector, the objective function in \eqref{EQ:NN} can easily be extended with an $\ell_1$-regularization to stimulate sparse solutions, since the vector $\veca$ appears affinely in $M(U,\veca,\vecb,\vecy)$:
\begin{equation}
	\min_{\veca \in \Complex^{n_a}} f(\veca) + \lambda \norm{\veca}_1, \label{eq:sparse}
\end{equation}
for some regularization parameter $\lambda$.

Note that for $\vecb = -\veca$,
\begin{equation}
	\norm{M(U,\veca,-\veca,\vecy)}_* = \norm{\vecy - \abs{U\veca}^2}_1 + n_y.
\end{equation}

Since the result of optimization \ref{EQ:NN} might not produce a desired solution sufficiently fitting the measurements, we propose the iterative Convex Optimization-based Phase Retrieval (COPR) algorithm, outlined in Algorithm~\ref{ALG:COPR}.

\begin{algorithm}
\caption{Convex Optimization-based Phase Retrieval (COPR)}\label{ALG:COPR}
\begin{algorithmic}[1]
\Procedure{COPR}{$\vecb,\tau$}\Comment{Some guess for $\vecb$}
\While{$\norm{\vecy - \abs{U\veca}}_1 > \tau$}\Comment{Termination criterion}
\State $\veca_+ \in \argmin_\veca \norm{M(U,\veca,\vecb,\vecy)}_*$
\State $\vecb_+\gets -\veca_+$
\EndWhile
\EndProcedure
\end{algorithmic}
\end{algorithm}

The nuclear norm minimization in Algorithm~\ref{ALG:COPR} is the main computational burden for an implementation.
Usual implementations of the nuclear norm involve semidefinite constraints, and require a semidefinite optimization solver.
If we assume that their computational complexity increases with $\bigO{n^6}$ \cite{vandenberghe2005interior} with constraint on matrices of size $n \times n$, then minimizing the nuclear norm of the matrix $M(U,\veca,\vecb,\vecy)$ of size $2n_y \times 2n_y$ is computationally infeasible even for
relatively
small-scale problems.
Therefore, we propose a tailored ADMM algorithm of which the computational complexity of the iterations scales $\bigO{n_y n_a}$, and requires the inverse of a matrix of size $2n_a \times 2n_a$ for every iteration of Algorithm \ref{ALG:COPR}.

\section{Efficient computation of the solution to \eqref{EQ:NN}}\label{sec:admm}

The minimization problem \eqref{EQ:NN} can be reformulated as:
\begin{align}\label{eq:admmopt}
\underset{X,\veca}{\min}\;
 \norm{X}_*\quad  \mbox{ subject to }\quad
X = M(U,\veca,\vecb,\vecy).
\end{align}

Applying the ADMM optimization technique \cite{boyd2011distributed} to the constraint optimization problem \eqref{eq:admmopt}, we obtain the steps in Algorithm~\ref{alg:admm}.

\begin{algorithm}
\caption{An ADMM algorithm for solving \eqref{eq:admmopt}}\label{alg:admm}
\begin{algorithmic}[1]
\Procedure{NN-ADMM}{$\vecb,\vecy,\rho,\tau$}
  \State $\veca \gets -\vecb$
  \State $\aX \gets M(U,\veca,\vecb,\vecy)$
  \State $\aY \gets 0 $
  \While{$\abs{\norm{M(U,\veca_{+},\vecb,\vecy)}_* - \norm{M(U,\veca,\vecb,\vecy)}_*} > \tau$}
    \State $\veca_{+} \in $
      \begin{equation}
          \underset{\veca}{\argmin} \norm{\aX - M(U,\veca,\vecb,\vecy) + \frac{1}{\rho}\aY}_F^2 \label{eq:aupdate}
      \end{equation}
    \State $\aX_{+} \in $
      \begin{equation}
        \underset{\aX}{\argmin} \norm{\aX}_* + \dfrac{\rho}{2}\norm{\aX - M(U,\veca_{+},\vecb,\vecy) + \dfrac{1}{\rho}\aY}_F^2 \label{eq:Xupdate}
      \end{equation}
    \State $\aY_{+} \gets \aY + \rho\left(\aX_{+} - M(U,\veca_{+},\vecb,\vecy)\right)$
    \State update $\rho$ according to the rules in \cite{boyd2011distributed}
  \EndWhile
  \EndProcedure
\end{algorithmic}
\end{algorithm}

The advantage of using this ADMM formulation is that both of the update steps \eqref{eq:aupdate} and \eqref{eq:Xupdate} have solutions that can be computed analytically.
The efficient computation of the solutions are described in the following two subsections.

\subsection{Efficient computation of the solution to \eqref{eq:aupdate}}\label{sec:aupdate}

Upon inspection of \eqref{eq:aupdate}, we see that this is a complex-valued standard least squares problem since $M(U,\veca,\vecb,\vecy)$ is parameterized affinely in $\veca$.
Let $\re{\cdot}$ and $\im{\cdot}$ respectively denote the real and the imaginary parts of a complex object.
Let the subscripts $(\cdot)_1$, $(\cdot)_2$ and $(\cdot)_3$ respectively denote the top-left, top-right and bottom-left submatrices according to \eqref{eq:M}.
Define 
\begin{equation*}
	\aZ = \aX + \dfrac{1}{\rho}\aY, \quad X = \dia{b^HU^H}.
\end{equation*}
In the sequel, let $\adi{P}$ denote the vector with the diagonal entries of a square matrix $P$.

Reordering the elements in \eqref{eq:aupdate}, separating the real and the imaginary parts, removing all matrix elements in the argument of the Frobenius norm that do not depend on $\veca$, and vectorizing the result, give the following least squares problem:
\begin{equation}
	\min_\vecx \norm{\vecu_{ADMM} - \vecu_{COPR} - AB\vecx}_2^2. \label{eq:ls}
\end{equation}
The variables $\vecu_{ADMM},~\vecu_{COPR},~A,~B$ and $\vecx$ are given by
\begin{equation}
  \begin{aligned}
      &\vecu_{ADMM} = \pmat{\adi{\re{\aZ_{1}}} \\ \adi{\re{\aZ_{2}}} \\ \adi{\re{\aZ_{3}}} \\ \adi{\im{\aZ_{2}}} \\ \adi{\im{\aZ_{3}}} },
      &&& &\vecu_{COPR} = \pmat{\vecy
      +\adi{\abs{X}^2} \\ \adi{\re{X}} \\ \adi{\re{X}} \\ \adi{\im{X}} \\ -\adi{\im{X}}}, \\
      &A = \pmat{2\re{X} & 2\im{X} \\ I & 0 \\ I & 0 \\ 0 & I \\ 0 & -I},
      &&& &B = \pmat{\re{U} & - \im{U} \\ -\im{U} & -\re{U}},
  \end{aligned}
\end{equation}
and $\vecx = \pmat{\re{\veca}^T & \im{\veca}^T}^T$.
This means that the optimal solution to \eqref{eq:ls} is given by 
\begin{equation*}
	\vecx^* = (B^TA^TAB)^{-1}B^TA^T(\vecu_{ADMM} - \vecu_{COPR}).
\end{equation*}
During the ADMM iterations only $\vecu_{ADMM}$ changes. The inverse $ (B^TA^TAB)^{-1} $ has to be computed once for every iteration of Algorithm~\ref{ALG:COPR} (i.e. it remains constant throughout the ADMM iterations).
Since the complexity of computing an inverse is $\bigO{n^3}$ for matrices of size $n \times n$, the computational complexity of this inverse process scales cubically with the number of basis functions.

Once this inverse matrix is obtained, the optimal solution to the least squares problem in \eqref{eq:ls} can be computed by a simple matrix-vector multiplication, whose complexity scales with $\bigO{n_yn_a}$.

Note that in the case that the objective term includes regularization as in \eqref{eq:sparse}, the optimization \eqref{eq:ls} should be modified appropriately to include the additive regularization term $\lambda\norm{\veca}_1$.

\subsection{Efficient computation of the solution to \eqref{eq:Xupdate}}\label{sec:Xupdate}
The optimization in \eqref{eq:Xupdate} is of the form
\begin{equation}
\underset{X}{\argmin} \norm{X}_* + \lambda\norm{X-C}_F^2. \label{eq:Xupdate_simple}
\end{equation}
Let $C= U_C\Sigma_CV_C^T$ be the singular value decomposition of $C \in \Complex^{2n_y \times 2n_a}$.
\begin{lemma}\label{lem:singular_vectors}
The solution $\aX$ to \eqref{eq:Xupdate_simple} has singular vectors $U_C$ and $V_C$.
\end{lemma}
\begin{proof}
Let
$X = U_X\Sigma_XV_X^T$ be a singular value decomposition of $X$.
Then
\begin{equation*}
	\begin{aligned}
		\norm{X}_* + \lambda\norm{X-C}_F^2 &= \trace{\Sigma_X} + \\
        &\qquad \lambda\left(\inp{X}{X} + \inp{C}{C} -2\inp{X}{C}\right).
	\end{aligned}
\end{equation*}
Using Von Neumann's trace inequality we get
\begin{equation*}
	\begin{aligned}
		& \min_{X} \paren{ \trace{\Sigma_X} + \lambda\left(\inp{X}{X} + \inp{C}{C} -2\inp{X}{C}\right)} \\
        \geq\; &\min_{X} \paren{\trace{\Sigma_X} + \lambda\left(\inp{X}{X} + \inp{C}{C} -2\trace{\Sigma_X\Sigma_C}\right)}\\
	\end{aligned}
\end{equation*}
with equality holds true when $C$ and $X$ are simultaneously unitarily diagonalizable.
The optimal solution $\aX$ to \eqref{eq:Xupdate_simple} therefore has the same singular vectors as $C$, i.e. $U_{\aX}=U_C,~ V_{\aX}=V_C$.
\end{proof}

Denote the singular values of $C$ in descending order as $\sigma_{C,1},\ldots,\sigma_{C,2n_y}$, and those of $X$ similarly.
Thanks to Lemma~\ref{lem:singular_vectors}, \eqref{eq:Xupdate_simple} can be simplified to
\begin{equation}\label{prob:sigma}
	\underset{\sigma_{X,i}}{\argmin} \sum_{i=1}^{2n_y}
    \paren{\sigma_{X,i} + \lambda\left(\sigma_{X,i} - \sigma_{C,i}\right)^2}.
\end{equation}

This problem is completely decoupled in $\sigma_{X,i}$ and the optimal solution to \eqref{prob:sigma} is computed with
\begin{equation*}
	\sigma_{\aX,i} = \max\left(0, \sigma_{C,i} - \frac{1}{2\lambda}\right),\quad i = 1,\ldots,2n_y.
\end{equation*}
By row and column permutations, the matrix $C$ is block-diagonal with blocks of size $2 \times 2$.
The SVD of this permuted matrix therefore involves block-diagonal matrices $U_C$, $\Sigma_C$ and $V_C$ and these blocks can be obtained separately and in parallel. Since the blocks are of size $2 \times 2$, the SVD can be obtained analytically.

This shows that a valid SVD can be computed very efficiently, in $\bigO{1}$.
That is, in theory, in a computation time independent of the number of pixels in the image, the number of images taken or of the number of basis functions.

\section{Convergence analysis of Algorithm \ref{ALG:COPR}}\label{sec:convergence}

Algorithm \ref{ALG:COPR} can be reformulated as a Picard iteration $\mathbf{a}_{k+1} \in T(\mathbf{a}_k)$, where the fixed point operator $T:\mathbb{C}^{n_a}\to \mathbb{C}^{n_a}$ is given by
\begin{equation}\label{T:operator}
  T(\mathbf{a}) = \arg\min_{\substack{\mathbf{x}\in \mathbb{C}^{n_a}}}\;\norm{M(U,\mathbf{x},-\mathbf{a},\mathbf{y})}_*.
\end{equation}

Our subsequent analysis will show that the set of fixed points, $\Fix T$, of $T$ is in general nonconvex and as a result, iterations generated by $T$ can not be \emph{Fej\'er monotone} \cite[Definition 5.1 of]{BauCom11} with respect to $\Fix T$.
Therefore, the widely known convergence theory based on the properties of \emph{Fej\'er monotone operators} and \emph{averaging operators} is not applicable to the operator $T$ given at \eqref{T:operator}.

In this section, we make an attempt to prove convergence of Algorithm \ref{ALG:COPR}, which has been observed from our numerical experiments, via a relatively new developed convergence theory based on the theory of \emph{pointwise almost averaging operators} \cite{LukNguTam16}.
It is worth mentioning that we are not aware of any other analysis schemes addressing convergence of Picard iterations generated by general \emph{nonaveraging} fixed point operators.
Our discussion consists of two stages. Based on the convergence theory developed in \cite{LukNguTam16}, we first formulate a convergence criterion for Algorithm \ref{ALG:COPR} (Proposition \ref{p:convergence}) under rather abstract assumptions on the operator $T$.
Due to the highly complicated structure of the nuclear norm of a general complex matrix, we are unable to verify these mathematical conditions for general matrices $U$.
However, we will verify that they are well satisfied in the case that $U$ is a unitary matrix (Theorem \ref{T:MATRIX_GENERAL}).
From the latter result, we heuristically hope that Algorithm \ref{ALG:COPR} still enjoys the convergence result when the matrix $U$ is close to being unitary in a certain sense.

It is a common prerequisite for analyzing local convergence of a fixed point algorithm that the set of solutions to the original problem is nonempty.
That is, there exists $\mathbf{a}\in \mathbb{C}^{n_a}$ such that $\mathbf{y}=|U\mathbf{a}|^2$.
Before stating the convergence result, we need to verify that the fixed point set of $T$ is nonempty.
\begin{lemma}\label{LEM:FIX_NONEMPTY}
The fixed point operator $T$ defined at \eqref{T:operator} holds
\[
\parennn{\mathbf{a} \mid \mathbf{y} = |U\mathbf{a}|^2} \;\subseteq\; \Fix T := \left\{\mathbf{a}\in \mathbb{C}^{n_a}\mid \mathbf{a}\in T(\mathbf{a})\right\}.
\]
\end{lemma}
\begin{proof}
See Appendix~\ref{APP:FIX_NONEMPTY}
\end{proof}

The next proposition provides an abstract convergence result for Algorithm \ref{ALG:COPR}.
$\Fix T$ is supposed to be closed.

\begin{proposition}\label{p:convergence}\cite[simplified version of Theorem 2.2 of]{LukNguTam16}
Let $S\subset \Fix T$ be closed with $T(\mathbf{a}^*) \subset\Fix T$ for all $\mathbf{a}^* \in S$ and let $W$ be a neighborhood of $S$. Suppose that $T$ satisfies the following conditions.
\begin{enumerate}
  \item[(i)]\label{t:subfirm convergence a} $T$ is \emph{pointwise averaging} at every point of $S$ with constant $\alpha\in (0,1)$ on $W$.
      That is, for all $\mathbf{a}\in W$, $\mathbf{a}_+\in T(\mathbf{a})$, $\mathbf{a}^*\in P_S(\mathbf{a})$ and $\mathbf{a}^*_+\in T(\mathbf{a}^*)$,
\begin{align}\label{averaged of T}
\norm{\mathbf{a}_+ - \mathbf{a}^*_+}^2 \le \norm{\mathbf{a}-\mathbf{a}^*}^2 - \frac{1-\alpha}{\alpha}\norm{(\mathbf{a}_+ - \mathbf{a})-(\mathbf{a}^*_+ - \mathbf{a}^*)}^2.
\end{align}
  \item[(ii)]\label{t:subfirm convergence b} The set-valued mapping $\psi:= T-\Id$ is \emph{metrically subregular} on $W$ for $0$ with constant $\gamma >0$, where $\Id$ is the Identity mapping.
  That is,
\begin{equation}\label{met_subreg}
\gamma\dist(\mathbf{a},\psi^{-1}(0)) \le \dist(0,\psi(\mathbf{a})),\quad \forall \mathbf{a}\in W.
\end{equation}
  \item[(iii)]\label{t:tech_assump} It holds $\dist(\mathbf{a},S) \le \dist(\mathbf{a},\Fix T)$ for all $\mathbf{a}\in W$.
\end{enumerate}
Then all Picard iterations $\mathbf{a}_{k+1}\in T(\mathbf{a}_k)$ starting in $W$ satisfy $\dist(\mathbf{a}_k,S)\to 0$ as $k\to \infty$ at least linearly.
\end{proposition}

Condition $(iii)$ in Proposition \ref{p:convergence} is, on one hand, a technical assumption and becomes redundant when $S=\Fix T$.
On the other hand, the set $S$ allows one to exclude from the analysis possible \emph{inhomogeneous} fixed points of $T$, at which the algorithm often exposes weird convergence behavior \cite[see Example 2.1 of]{LukNguTam16}.

The size of neighborhood $W$ appearing in Proposition \ref{p:convergence} indicates the robustness of the algorithm in terms of erroneous input (the distance from the starting point to a nearest solution).

We now apply the abstract result of Proposition \ref{p:convergence} to the following special, but important case.

\begin{thm}\label{T:MATRIX_GENERAL} 
Let $U\in \mathbb{C}^{n_a\times n_a}$ be unitary and $\mathbf{a}^*\in \mathbb{C}^{n_a}$ be such that $|U\mathbf{a}^*|^2=\mathbf{y}$.
Then every Picard iteration generated by Algorithm~\ref{ALG:COPR} $\mathbf{a}_{k+1}\in T(\mathbf{a}_k)$ starting sufficiently close to $\mathbf{a}^*$ converges linearly to a point $\tilde{\mathbf{a}} \in \Fix T$ satisfying $|U\tilde{\mathbf{a}}|^2=\mathbf{y}$.
\end{thm}
\begin{proof}
See Appendix~\ref{APP:MATRIX GENERAL}.
\end{proof}

\section{Numerical experiments}\label{sec:numericalexperiments}

Three important numerical aspects of the CORP algorithm, including flexibility, complexity, and robustness, are tested on relevant problems.
First, we demonstrate the flexibility of the convex relaxation by comparing the COPR algorithm with an added $\ell_1$-regularization to the PhaseLift method \cite{candes2013phaselift} and to the CPRL method in \cite{Ohlsson} on an under-determined sparse estimation problem.
Second, we compare the practically observed computational complexity of COPR and a 
naive implementation of PhaseLift \cite{candes2013phaselift}.
Finally, we investigate the robustness of CORP relative to noise
in a Monte-Carlo simulation for 25 and 100 basis functions. We compare four algorithms: COPR, PhaseLift \cite{candes2013phaselift}, a basic alternating projections method (Section 4.3 in \cite{candes2013phaselift}) and an averaged projections method based on \cite{Luke_Toolbox}.
We note that the latter method fundamentally employs the Fourier transform at every iteration and hence is, in generally, not applicable for phase retrieval in the modal form.

\subsection{Application of COPR to compressive sensing problems}

The first problem is to estimate 16 coefficients from 8 measurements, where the optimal vector is known to be sparse.

We generate a sparse coefficient vector $\veca$ with two randomly generated non-zero complex elements. We generate two images ($n_d = 2,~m = 128$) by applying two different amounts of defocus with Zernike coefficients $-\frac{\pi}{8}$ and $\frac{\pi}{8}$, respectively. From each image we use the center $2 \times 2$ pixels, resulting in a total of $n_y = 8$ measurements.

The applied algorithms are the COPR algorithm, the COPR algorithm with an additional $\ell_1$-regularization, the PhaseLift algorithm \cite{candes2013phaselift} and the Compressive sensing Phase Retrieval (CPRL) algorithm of \cite{Ohlsson}. The results are displayed in Figure~\ref{fig:sparse}.
\begin{figure}[ht]
	\centering
    \includegraphics[width=0.7\columnwidth]{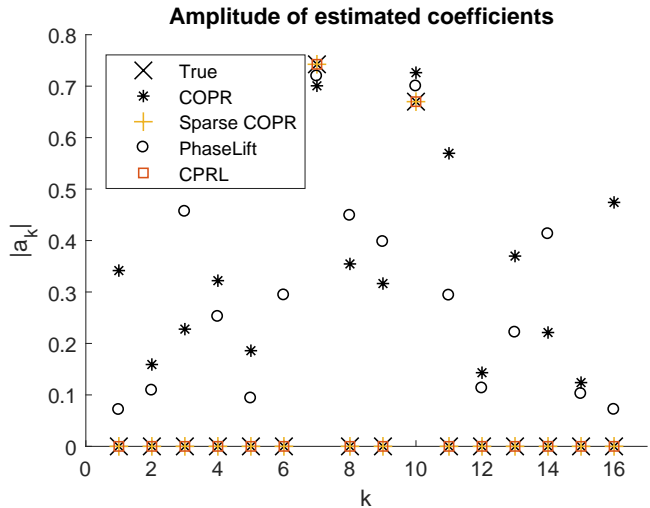}
    \caption{The absolute values of 16 estimated coefficients according to 4 different algorithms.}
    \label{fig:sparse}
\end{figure}
As can be seen from the figure, COPR and PhaseLift fail to retrieve the correct solution. The CPRL method and the regularized COPR algorithm compute the correct solution.


\subsection{Computational complexity}

The second problem demonstrates the trends of the required computation time when the number of estimated coefficients increases.
The underlying estimation problem consists of 7 images with different amounts of defocus applied as phase diversity, where each image is of size 128 by 128 pixels. A subset of 20 by 20 pixels of each image is used in the estimation.
We compare the COPR algorithm to the PhaseLift algorithm, which is implemented according to optimization problem (2.5) in \cite{candes2013phaselift}.
\begin{figure}[ht]
	\centering
    \includegraphics[width=0.7\columnwidth]{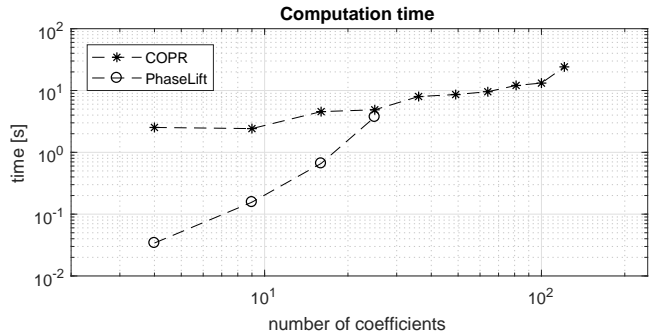}
    \caption{A computation time comparison between PhaseLift and COPR for different numbers of coefficients.}
    \label{fig:trend}
\end{figure}
For PhaseLift, the reported time is the time it takes the MOSEK solver \cite{mosek} to solve the optimization problem. This does not include the time taken by YALMIP \cite{Lofberg2004} to convert the problem as given to the solver-specific form.
For COPR, the number of iterations is set beforehand according to convergence to the correct solution, and the total time is recorded.
By convergence we mean that the estimated vector $\hat{\veca}$ satisfies the tolerance criterion:
\begin{equation}
\min_{c \in \Complex,~\abs{c} = 1} \norm{c\hat{\veca} - {\veca}^*}_2^2 \leq 10^{-5},
\end{equation}
where $\veca^*$ is the exact solution.

The minimization over the parameter $c$ ensures that the (unobservable) piston mode in the phase is canceled.\footnote{Let $\pmat{\hat{\veca} & {\veca}^*} = QR$ be the QR decomposition. Then $\angle c^* = \angle \frac{R_{12}}{R_{11}}$. }
The computational complexity of PhaseLift is, as implemented, approximately $\bigO{n^4}$. The MOSEK solver ran into numerical issues for more than 25 estimated parameters.
The COPR algorithm's computational complexity is approximately $\bigO{n}$. The better complexity is offset by a longer computation time for very small problems.

\subsection{Robustness to noise}
When estimating of an unknown phase aberration, it is more logical to evaluate the performance of the algorithm on its ability to estimate the phase, and not the coefficients of basis functions.

We assume the phase is randomly generated with a deformable mirror. Let $H \in \Real^{m^2 \times n_u}$ be the mirror's influence matrix and $\vecu \in \Real^{n_u}$ be the input to the mirror's actuators, such that
\begin{equation}
	\phi_{DM} = H \vecu. \label{eq:DM}
\end{equation}

The input values $u_i$ are drawn from the uniform distribution between 0 and 1.
The mirror has $n_u = 44$ actuators and the images have sides $m = 128$. The aperture radius is $0.4$.

Five different defocus diversities are applied with Zernike coefficients uniformly spaced between $-\frac{\pi}{2}$ and $\frac{\pi}{2}$.
Gaussian noise is added to the obtained images such that
\begin{equation}
	\vecy = \max(0, \abs{\fourier{P_d(\rho,\theta)}}^2+ \varepsilon),~ \varepsilon \in N(0,\sigma I).
\end{equation}
and $\sigma$ is the noise variance.
No denoising methods were applied.
The signal-to-noise ratio (SNR) is computed according to
\begin{equation}
	10 \log_{10} \frac{\norm{\vecy - \abs{\fourier{P_d(\rho,\theta)}}^2 }_2^2}{\norm{\abs{\fourier{P_d(\rho,\theta)}}^2}_2^2}.
\end{equation}

The phase is estimated from $\vecy$ using four different algorithms.
The first is the COPR algorithm.
The second is the averaged projections (AvP) algorithm \cite{Luke_Toolbox}.
The third is the alternating projections (AlP) method (\cite{candes2013phaselift}, section 4.3), and the fourth algorithm is the PhaseLift method \cite{candes2013phaselift}.

The COPR and the AlP methods are applied for two cases corresponding to using 25 and 100 basis functions.
The PhaseLift method is applied for only the case with 25 basis functions due to numerical problems in the solver for larger problems.

The AvP method is not based on the use of basis functions but on the Fourier transform.
Due to the sensitivity to noise of this method, 100 basis functions were fit to the estimated object plane field.
The phase generated by these weighted basis functions was used to report performance.
The use of basis functions improved the phase estimate.

We make use of the Strehl ratio as a measure of optical quality.
The Strehl ratio $S$ is the ratio of the maximum intensity of the aberrated PSF and that of the unaberrated one and can be approximated with the expression of Mahajan:
\begin{equation*}
	S \approx \expp{-\delta^2},
\end{equation*}
where $\delta = \norm{\phi_{DM} - \hat{\phi}}_2$ and the mean residual phase has been removed \cite{roddier1999adaptive}.

For every noise level, 100 different phases were generated with the deformable mirror model \eqref{eq:DM}.
The results are presented in Figure~\ref{fig:noise}.
\begin{figure}[ht]
	\centering
    \includegraphics[width=1\columnwidth]{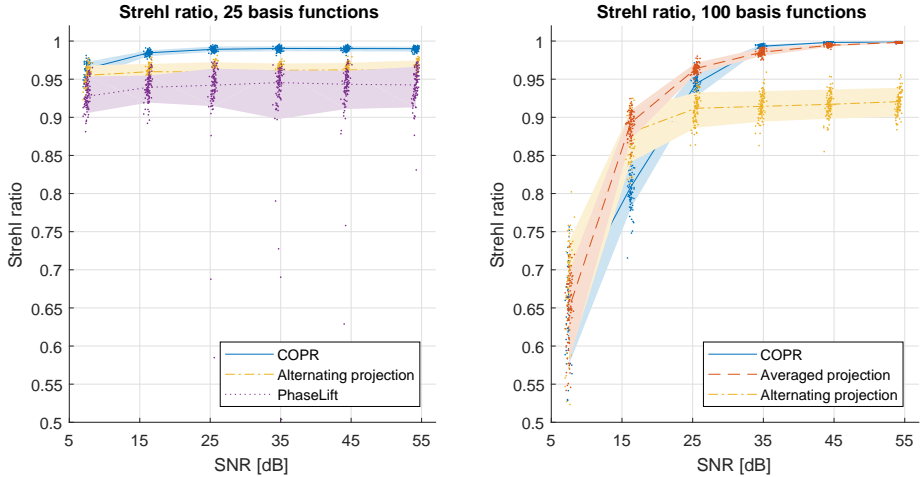}
    \caption{The Strehl ratio of the estimated phase aberration as a function of SNR. The shaded areas indicate the 10\% and 90\% quantiles.}
    \label{fig:noise}
\end{figure}
The resulting Strehl-ratio's are plotted with a trend line and shaded quantile lines at 10\% and 90\%.

In the case of PhaseLift, the tuning parameter that trades off measurement fit and the rank of the `lifted' matrix is tuned once and applied to all problems.
This has the effect that the reported performance is not as high as it could be with optimal tuning for individual problems.
This points to another advantage of COPR: the absence of tuning parameters aside from the choice of basis functions.

The figure shows that COPR appears to be robust to noise.
Also, the figure on the right shows that when the number of basis functions is high, the estimated phase is very close to the exact phase in low noise settings, something that cannot be done with 25 basis functions.
However, when the noise level is high, the choice for a smaller number of basis functions shows better performance.
We attribute this to overfitting in high noise level circumstances.

\section{Concluding Remarks}\label{sec:remarks}
The convex relaxations in solving the phase retrieval problem as proposed in \eqref{EQ:NN} have the advantage over current convex relaxation methods, such as PhaseLift, that our strategy is affine in the coefficients that are to be estimated.
This allows for easy extension of the proposed method to phase retrieval problems that incorporate prior knowledge on the coefficients by regularization of the objective function.
One such successful extension is the regularization with the $\ell_1$-norm to find sparse solutions, as demonstrated in Figure~\ref{fig:sparse}.

In Section~\ref{sec:admm} an ADMM algorithm was proposed for efficient computation of the solution to \eqref{EQ:NN}. 
The result is that for the COPR algorithm a better computational complexity is observed compared to PhaseLift, see Figure~\ref{fig:trend}. 
COPR is also able to solve phase estimation problems with larger numbers of parameters.

The required computations are favourable both in computation time and accuracy (they have simple analytic solutions) and in worst-case scaling behaviour $\bigO{ n_y n_a}$ for every ADMM iteration, where $n_y$ is the number of pixels and $n_a$ is the number of basis functions.

We discussed convergence properties of the COPR algorithm in Section~\ref{sec:convergence} and showed that for selected problems this convergence is linear or faster.

Finally, COPR has been shown to be robust against measurement noise, and outperform the two projection-based methods whose naive forms are often sensitive to noise as expected.

We are aware that in practice the performance of projection methods can be substantially better than what we have observed in this study provided that appropriate denoising techniques are also applied.
Keeping aside from the matter of using denoising techniques, we have chosen to compare the algorithms in their very definition forms.

\section{Funding Information}

The research leading to these results has received funding from the European Research Council under the European Union's Seventh Framework Programme (FP7/2007-2013) / ERC grant agreement No. 339681.

\bibliographystyle{ieeetr}
\bibliography{bibliography,TUDelft_reference}

\appendix
\section{Proof of Lemma~\ref{LEM:FIX_NONEMPTY}}\label{APP:FIX_NONEMPTY}

\begin{proof}
Let $\mathbf{a}$ satisfy $\mathbf{y}=|U\mathbf{a}|^2$. It suffices to check that $\mathbf{a} \in T(\mathbf{a})$.
We first observe that
\[
\ra\paren{M(U,\mathbf{a},-\mathbf{a},\mathbf{y})} = \ra {\pmat{ 0 & 0 \\ 0 & I_{n_y} }}
 = n_y.
\]
This means that $\mathbf{a}$ is a global minimizer of $\ra M(U,\mathbf{x},-\mathbf{a},\mathbf{y})$ as a function of $\mathbf{x}\in \mathbb{C}^{n_a}$.
Since the nuclear norm
$\norm{M(U,\mathbf{x},-\mathbf{a},\mathbf{y})}_*$ is the \emph{convex envelop} of the $\ra M(U,\mathbf{x},-\mathbf{a},\mathbf{y})$, they have the same global minimizers.
Hence, $\mathbf{a}$ is also a global minimizer of $\norm{M(U,\mathbf{x},-\mathbf{a},\mathbf{y})}_*$ as a function of $\mathbf{x}$, that is
\[
  \mathbf{a}\in \arg\min_{\mathbf{x}\in \mathbb{C}^{n_a}}\; \norm{M(U,\mathbf{x},-\mathbf{a},\mathbf{y})}_*.
\]
In other words, $\mathbf{a}\in T(\mathbf{a})$ and the proof is complete.
\hfill
\end{proof}

\section{Proof of Theorem~\ref{T:MATRIX_GENERAL}}\label{APP:MATRIX GENERAL}
Lemma~\ref{t:matrix I} will serve as the basic step for proving Theorem \ref{T:MATRIX_GENERAL}.

\begin{lemma}\label{t:matrix I} Let $U=I_{n_a}$ and $\mathbf{a}^*\in \mathbb{C}^{n_a}$ be such that $|U\mathbf{a}^*|^2=\mathbf{y}$.
Then every Picard iteration $\mathbf{a}_{k+1}\in T(\mathbf{a}_k)$ starting sufficiently close to $\mathbf{a}^*$ converges linearly to a point $\tilde{\mathbf{a}} \in \Fix T$ satisfying $|U\tilde{\mathbf{a}}|^2=\mathbf{y}$.
\end{lemma}

\begin{proof}
Since $U=I_{n_a}$, the nuclear norm of $M(I_{n_a},\mathbf{x},-\mathbf{a},\mathbf{y})$ can be calculated from the nuclear norms of $n_a$ matrices $M(1,x_i,-a_i,y_i) \in \mathbb{C}^{2\times 2}$ $(1\le i\le n_a)$.
Let us do the calculation for an arbitrary $\mathbf{a}\in \mathbb{C}^{n_a}$.
We first calculate the nuclear norm of each $2\times 2$ matrix
\begin{equation*}
  M(1,x_i,-a_i,y_i) = \left(
                        \begin{array}{cc}
                          y_i-2\re{x_i\overline{a_i}}+ |a_i|^2 & x_i-a_i\\
                          \overline{x_i}-\overline{a_i} & 1\\
                        \end{array}
                      \right).
\end{equation*}
Indeed, we have by direct calculation that
\begin{equation}\label{f_i}
	\begin{aligned}
   		f_i(x_i) &:= \norm{M(1,x_i,-a_i,y_i)}_*^2  \\
        &= \norm{\pmat{ r & s \\ s & 1 } }_*^2  \\
        &= r^2 + 2s^2 + 1 + 2|r-s^2|,
   \end{aligned}
\end{equation}
where
\begin{equation*}
  r := y_i-2\re{x_i\overline{a_i}}+ |a_i|^2,\quad s := |x_i-a_i|.
\end{equation*}

Let us denote
\begin{equation}\label{T_i_def}
  T_i(a_i) := \arg\min_{\substack{x_i\in \mathbb{C}}}\; f_i(x_i).
\end{equation}

Solving analytically the minimization problem on the right-hand side of \eqref{T_i_def}, we obtain the explicit form of $T_i$ as follows
\begin{equation}\label{T_i}
  T_i(a_i) =
  \left\{
    \begin{array}{ll}
      \left\{z \in \mathbb{C} \mid |z| \le \sqrt{y_i}\right\}, & \hbox{if }\quad a_i=0,\\
      \left\{\frac{\sqrt{y_i}}{|a_i|}a_i\right\}, & \hbox{if }\quad 0< |a_i| \le \sqrt{\lambda_i},\\
      \left\{\frac{y_i+|a_i|^2+1}{2(|a_i|^2+1)}a_i\right\}, & \hbox{if }\quad |a_i|\ge \sqrt{\lambda_i},
    \end{array}
  \right.
\end{equation}
where $\lambda_i$ is the unique real positive root of the real polynomial $g_i(t):=t^3+2(1-y_i)\,t^2+(y_i^2-6y_i+1)t-4y_i$.

We need to take care of the two possible cases of $y_i$.

\textbf{Case 1.} $y_i\in (0,1]$. Then we have $\frac{3}{2}\sqrt{y_i} < \sqrt{\lambda_i} <2\sqrt{y_i}$ since $g_i\paren{\frac{9}{4}y_i} <0$ and $g_i\paren{4y_i} >0$. The following properties of $T_i$ can be verified.
      \begin{itemize}
  \item $\Fix T_i = \left\{z \in \mathbb{C} \mid |z| = \sqrt{y_i}\right\} \cup \{0\}$, where $0$ is an inhomogeneous fixed point of $T_i$, that is, $T_i(0) \nsubseteq \Fix T_i$.
  \item The set of homogeneous fixed points of $T_i$ is $S_i:=\left\{z \in \mathbb{C} \mid |z| = \sqrt{y_i}\right\}$.
  \item $T_i$ is pointwise
  averaging at every point of $S_i$ on $W_i := \{z\in \mathbb{C}\mid |z| \ge \sqrt{y_i}/2\}$ with
  constant $3/4$.
  \item The set-valued mapping $\psi_i:= T_i-\Id$ is metrically subregular on $W_i$ for $0$ with constant $1/2$.
  \item The technical assumption $\dist(z,S_i) \le \dist(z,\Fix T_i)$ holds for all $z\in W_i$.
\end{itemize}

\textbf{Case 2.} $y_i=0$. Then $\lambda_i=0$. Note also that $a^*_i=0$ and the formula \eqref{T_i} becomes
$T_i(a_i) = \frac{1}{2}a_i$. The following properties of $T_i$ can be verified.
      \begin{itemize}
  \item $\Fix T_i = \{0\}$, where $0$ is a homogeneous fixed point of $T_i$.
  \item $T_i$ is pointwise
  averaging at every point of $S_i$ on $\mathbb{C}$ with
  constant $1/4$.
  \item The set-valued mapping $\psi_i:= T_i-\Id$ is metrically subregular on $\mathbb{C}$ for $0$ with constant $1/2$.
  \item The technical assumption $\dist(z,S_i) \le \dist(z,\Fix T_i)$ holds for all $z\in \mathbb{C}$.
\end{itemize}
 In this case, we denote $S_i:=\{0\}$ and $W_i:=\mathbb{C}$.
 \bigskip

The operator $T$ can be calculated explicitly
\begin{equation}\label{eg:operator_matrix}
T(\mathbf{a}) = \arg\min_{\substack{\mathbf{x}\in \mathbb{C}^{n_a}}}\; \sum_{i=1}^{n_a} \sqrt{f_i(x_i)},\quad \forall \mathbf{a}\in \mathbb{C}^{n_a},
\end{equation}
where the constituent functions $f_i(x_i)$ are given by \eqref{f_i}.

Minimizing $f_i$ ($i=1,2\ldots,n_a$) separately yields the explicit form of $T$ as a Cartesian product
\begin{equation}\label{T(a) product}
  T(\mathbf{a}) = T_1(a_1) \times T_2(a_2) \cdots \times T_{n_a}(a_{n_a}),
\end{equation}
where the component operators $T_i$ are given by \eqref{T_i}.

Thanks to the separability structure of $T$ as a Cartesian product at \eqref{T(a) product}, the following properties of $T$ in relation to Proposition \ref{p:convergence} can be deduced from the corresponding ones of the component operators $T_i$.
\begin{itemize}
  \item $\Fix T = \prod_{i=1}^{n_a} \Fix T_i$ and the set of homogeneous fixed points of $T$ is $S:=\prod_{i=1}^{n_a} S_i$. It is clear that $|U{\mathbf{a}}|^2=\mathbf{y}$ for $U=I_{n_a}$ and all ${\mathbf{a}}\in S$.
  \item $T$ is pointwise
  averaging at every point of $S$ on $W:=\prod_{i=1}^{n_a} W_i$ with
  constant $\alpha=3/4$.
  \item The set-valued mapping $\psi:= T-\Id$ is \emph{metrically subregular} on $W$ for $0$ with constant $\kappa=1/2$.
  \item The technical assumption 
  $(iii)$ of Proposition \ref{p:convergence} is satisfied on $W$. That is,
      \begin{equation}\label{tech_ass_T}
        \dist(\mathbf{w},S) \le \dist(\mathbf{w},\Fix T),\quad \forall \mathbf{w}\in W.
      \end{equation}

\end{itemize}
Now we can apply Proposition \ref{p:convergence} to conclude that every Picard iteration $\mathbf{a}_{k+1}\in T(\mathbf{a}_k)$ starting in $W$ converges linearly
to a point in $S$ as claimed.
\hfill
\end{proof}

\begin{remark}
Under the assumption that $y_i>0$ for all $1\le i\le n_a$, then the linear convergence result established in Lemma~\ref{t:matrix I} can be sharpened to finite convergence.
\end{remark}

In order to distinguish the fixed point operator \eqref{T:operator} corresponding to a general unitary matrix $U$ from the one analyzed in Lemma~\ref{t:matrix I} corresponding to the identity matrix $I_{n_a}$, in the following proof, we will use the notation $\widehat{T}$ for one specified in Theorem~\ref{T:MATRIX_GENERAL}.

\begin{proof}
Let $T$ be the fixed point operator \eqref{T:operator} which corresponds to the identity matrix and has been analyzed in Lemma~\ref{t:matrix I}.
We start the proof by proving that
\begin{equation}\label{swap}
  \widehat{T}(\mathbf{a}) = U^{-1}T(U\mathbf{a}),\quad \forall \mathbf{a}\in\mathbb{C}^{n_a}.
\end{equation}
Indeed, let us take an arbitrary $\mathbf{a}\in \mathbb{C}^{n_a}$ and denote $\mathbf{a}' = U\mathbf{a}$. Then we have
\begin{equation}
  \begin{aligned}
      \widehat{T}(\mathbf{a}) &= \arg\min_{\mathbf{x}\in\mathbb{C}^{n_a}}
      \; \norm{M(U,\mathbf{x},-\mathbf{a},\mathbf{y})}_*
      \\
      &= \arg\min_{\mathbf{x}\in\mathbb{C}^{n_a}}
      \; \norm{M(I_{n_a},U\mathbf{x},-\mathbf{a}',\mathbf{y})}_*
      \\
      &= U^{-1}\paren{\arg\min_{\mathbf{x}\in\mathbb{C}^{n_a}}\; \norm{M(I_{n_a},\mathbf{x},-\mathbf{a}',\mathbf{y})}_*}
      \\
      &= U^{-1}\paren{T(\mathbf{a}')} = U^{-1}\paren{T(U\mathbf{a})}.
  \end{aligned}
\end{equation}
We have proved \eqref{swap}. As a consequence,
\begin{equation}
  \begin{aligned}\label{Fix T_hat}
      \Fix \widehat{T} &= \{\mathbf{a}\in \mathbb{C}^{n_a} \mid \mathbf{a}\in \widehat{T}(\mathbf{a})\} \\
      & = \{\mathbf{a}\in \mathbb{C}^{n_a} \mid \mathbf{a}\in U^{-1}T(U\mathbf{a})\} \\      
      &= \{\mathbf{a}\in \mathbb{C}^{n_a} \mid U\mathbf{a}\in T(U\mathbf{a})\} \\
      &= \{\mathbf{a}\in \mathbb{C}^{n_a} \mid U\mathbf{a}\in \Fix T\} = U^{-1}\paren{\Fix T}.
  \end{aligned}
\end{equation}
For the sets $S$ and $W$ determined in the proof of Lemma~\ref{t:matrix I}, we denote $\widehat{S}:=U^{-1}(S)$ and $\widehat{W}:=U^{-1}(W)$. Since $U$ is a unitary matrix, the set of homogeneous fixed points of $\widehat{T}$ is $\widehat{S}:=U^{-1}(S)$.
It also holds by the definition of projection and \eqref{Fix T_hat} that, for all $\mathbf{w}\in W$,
\begin{equation}
	P_{U^{-1}(S)}\paren{U^{-1}\mathbf{w}} = U^{-1}\paren{P_{S}(\mathbf{w})}, \label{technical}
\end{equation}
\begin{equation}
	\dist\paren{U^{-1}\mathbf{w},U^{-1}(S)} = \dist\paren{U^{-1}\mathbf{w},U^{-1}(\Fix T)}.\label{vincinity}
\end{equation}
We now can verify the three assumptions on $\widehat{T}$ imposed in Proposition \ref{p:convergence}.

\begin{itemize}
\item $\widehat{T}$ is pointwise
averaging at every point of $\widehat{S}$ on $\widehat{W}$ with
constant $\alpha=3/4$.

    Indeed, take an arbitrary $\mathbf{a}\in \widehat{W}$, $\mathbf{a}_+\in \widehat{T}(\mathbf{a})$, $\hat{\mathbf{a}}\in P_{\widehat{S}}(\mathbf{a})$ and $\hat{\mathbf{a}}_+\in \widehat{T}(\hat{\mathbf{a}})$.
By definition of $\widehat{W}$, there is $\mathbf{w}\in W$ such that $\mathbf{a}=U^{-1}\mathbf{w}$ and $\mathbf{a}_+ = U^{-1}T(\mathbf{w})$, and thanks to \eqref{vincinity}, $\hat{\mathbf{a}} = U^{-1}\paren{P_S(\mathbf{w})}$ and $\hat{\mathbf{a}}_+ = U^{-1}\paren{P_S(\mathbf{w})} = \hat{\mathbf{a}}$.
Then by the pointwise
averagedness of $T$ established in the proof of Lemma~\ref{t:matrix I}, we have

\begin{equation}
  \begin{aligned}
  	\norm{\mathbf{a}_+ - \hat{\mathbf{a}}_+}^2 &= \norm{U^{-1}T(\mathbf{w}) - U^{-1}\paren{P_S(\mathbf{w})}}^2 \\
  		&= \norm{T(\mathbf{w}) - P_S(\mathbf{w})}^2\\
   		&\le \norm{\mathbf{w}-P_S(\mathbf{w})}^2 - \frac{1}{3}\norm{T(\mathbf{w}) - \mathbf{w}}^2\\
   		&= \norm{U^{-1}\mathbf{w}-U^{-1}\paren{P_S(\mathbf{w})}}^2 \\
        & \qquad - \frac{1}{3}\norm{U^{-1}(T(\mathbf{w})) - U^{-1}\mathbf{w}}^2\\
   &= \norm{\mathbf{a}-\hat{\mathbf{a}}}^2 - \frac{1}{3}\norm{\mathbf{a}_+ - \mathbf{a}}^2
  \end{aligned}
\end{equation}
as claimed.

\item The set-valued mapping $\widehat{\psi}:= \widehat{T}-\Id$ is metrically subregular on $\widehat{W}$ for $0$ with constant $\gamma = 1/2$.

Indeed, take an arbitrary $\mathbf{a}\in \widehat{W}$. By definition of $\widehat{W}$, there is $\mathbf{w}\in W$ such that $\mathbf{a}=U^{-1}\mathbf{w}$.
Then by \eqref{Fix T_hat} and the metric subregularity of $\psi$ established in the proof of Lemma~\ref{t:matrix I}, we have
\begin{equation*}
	\begin{aligned}
		\dist\paren{\mathbf{a},\widehat{\psi}^{-1}(0)} &= \dist\paren{\mathbf{a},\Fix \widehat{T}} \\
		&= \dist\paren{U^{-1}\mathbf{w},U^{-1}\paren{\Fix T}} \\
		&= \dist\paren{\mathbf{w},\Fix T} \\
 		&\le \frac{1}{2}\dist\paren{\mathbf{w},T(\mathbf{w})} \\
        &= \frac{1}{2}\dist\paren{U^{-1}\mathbf{w},U^{-1}(T(U\mathbf{a}))} \\
 		&= \frac{1}{2}\dist\paren{\mathbf{a},\widehat{T}(\mathbf{a})} = \frac{1}{2}\dist\paren{0,\widehat{\psi}(\mathbf{a})}    
    \end{aligned}
\end{equation*}
as claimed.

\item The technical assumption 
$(iii)$ of Proposition \ref{p:convergence} is satisfied on $\widehat{W}$.

     Indeed, take an arbitrary $\mathbf{a}\in \widehat{W}$. By definition of $\widehat{W}$, there is $\mathbf{w}\in W$ such that $\mathbf{a}=U^{-1}\mathbf{w}$. Then by \eqref{tech_ass_T}, \eqref{Fix T_hat} and \eqref{technical}, we have
\begin{equation*}
	\begin{aligned}
		\dist\paren{\mathbf{a},\widehat{S}} &= \dist\paren{U^{-1}\mathbf{w},U^{-1}(S)} = \dist\paren{\mathbf{w},S} \\
    		&\le \dist\paren{\mathbf{w},\Fix T} = \dist\paren{U^{-1}\mathbf{w},U^{-1}(\Fix T)} \\
            &= \dist\paren{\mathbf{a},\Fix \widehat{T}}
    \end{aligned}
\end{equation*}
as claimed.
\end{itemize}

Therefore, we can apply Proposition \ref{p:convergence} to conclude that every Picard iteration $\mathbf{a}_{k+1}\in \widehat{T}(\mathbf{a}_k)$ generated by the COPR algorithm starting in $\widehat{W}$ converges linearly to a point $\tilde{\mathbf{a}}\in \widehat{S}$.
Finally, let $\widetilde{\mathbf{w}}\in S$ such that $\tilde{\mathbf{a}} = U^{-1}\widetilde{\mathbf{w}}$. It holds that
$|U\tilde{\mathbf{a}}|^2 = |\widetilde{\mathbf{w}}|^2 = \mathbf{y}$ by the structure of $S$.

The proof is complete.
 \hfill
\end{proof}

\end{document}